\documentclass[12pt]{amsart} 

\def\public{0}

\usepackage{amssymb} 
\usepackage{amsfonts} 
\usepackage{amsmath} 
\usepackage{amsthm} 
\usepackage{array} 
\usepackage{multirow}
\usepackage{geometry} 
\usepackage{graphicx,psfrag} 
\usepackage{mathrsfs} 
\usepackage[all]{xy} 
\usepackage{verbatim} 
\usepackage{stmaryrd}
\usepackage[colorlinks]{hyperref}
\usepackage{chngpage}
\usepackage{booktabs}
\usepackage[usenames,dvipsnames]{color}
\theoremstyle{plain} 
\newtheorem{thm}{Theorem}[section] 
\newtheorem{cor}[thm]{Corollary} 
\newtheorem{lemma}[thm]{Lemma}
\newtheorem{prop}[thm]{Proposition} 
\newtheorem*{thm*}{Theorem}

\theoremstyle{remark} 

\newcommand{\Z}{\mathbb{Z}}     
\newcommand{\Q}{\mathbb{Q}} 
\newcommand{\R}{\mathbb{R}} 
\newcommand{\LL}{\mathbb{L}}     
\newcommand{\PP}{\mathbb{P}}     %
\newcommand{\sC}{\mathcal{C}}
\newcommand{\sD}{\mathcal{D}}

\newcommand{\sM}{\mathcal{M}}
\newcommand{\sO}{\mathcal{O}}
\newcommand{\sR}{\mathcal{R}}

\newcommand{\sU}{\mathcal{U}}
\newcommand{\fc}{\mathfrak{c}}
\newcommand{\op}{\overline{p}}
\newcommand{\oC}{\overline{C}}
\newcommand{\bx}{\textbf{x}}

\newcommand{\beq}{\begin{equation}}
\newcommand{\eeq}{\end{equation}}

\newcommand{\steffen}[1]{\marg{\color{red} Steffen: \color{blue} #1}}

\newcommand{\cms}{\overline{\sM}} 
\newcommand{\relmaps}[2]{\cms_{#1}(\PP^1;#2)}
\newcommand{\relrub}[2]{\cms^\sim_{#1}(\PP^1;#2)}

\newcommand{\HH}[2]{\mathbb{H}_{#1}(#2)}     
\newcommand{\relrubU}[2]{\overline{\mathcal{U}}^\sim_{#1}(#2)}

\if\public1
\newcommand{\marg}[1]{}
\else
\newcommand{\marg}[1]{\normalsize{{\color{red}\footnote{{\color{blue}#1}}}{\marginpar[{\color{red}\hfill\tiny\thefootnote$\rightarrow$}]{{\color{red}$\leftarrow$\tiny\thefootnote}}}}}
\fi
\newcommand{\ch}{{\rm CH}}
\newcommand{\br}{{\rm br}}

\newcommand{\pt}{{\rm pt}}
\providecommand{\stab}{{\rm stab}}
\providecommand{\sst}{{\rm ss}}

\newcolumntype{S}{>{\centering\arraybackslash} m{3in} }
\newcolumntype{U}{>{\arraybackslash} m{2in} }
\newcolumntype{T}{>{\centering\arraybackslash} m{1 in} }

\begin{document}

%
%
\title{A geometric perspective on the piecewise polynomiality of double Hurwitz numbers}
\date{\today}
\author[R. Cavalieri]{Renzo Cavalieri}
\address{Renzo Cavalieri, Colorado State University, Department of Mathematics, Weber Building, Fort Collins, CO 80523, U.S.A}
\thanks{Research of R.C. supported by NSF grant DMS-1101549, NSF RTG grant 1159964 }
\email{renzo@math.colostate.edu}
\author[S. Marcus]{Steffen Marcus}
\address{Steffen Marcus, Department of Mathematics, University of Utah, 155 S 1400 E Room 233, Salt Lake City, UT 84112, U.S.A}
\email{marcus@math.utah.edu}
\begin{abstract}
We describe double Hurwitz numbers as intersection numbers on the moduli space of curves $\cms_{g,n}$. Assuming polynomiality of the Double Ramification Cycle (which is known in genera $0$ and $1$), our formula explains the polynomiality in chambers of double Hurwitz numbers, and the wall crossing phenomenon in terms of a variation of correction terms to the $\psi$ classes.  We interpret this as suggestive evidence for polynomiality of the Double Ramification Cycle.
\end{abstract}
\maketitle
\setcounter{tocdepth}{1}
\setcounter{table}{0}
%


\section{Introduction}

This article investigates the piecewise polynomiality of double Hurwitz numbers (see Section~\ref{sec:dhn} for background) from a geometric perspective. The combinatorial aspects of this theory have been extensively studied in \cite{gjv:dhn}, \cite{ssv:gz} and \cite{cjm:wc}, where the chambers of polynomiality and inductive wall crossing formulae are explicitly described. However, double Hurwitz numbers arise naturally from geometry, as the degree of  a zero dimensional cycle: the pullback of a point via the branch map from a compactification of a Hurwitz space to the corresponding moduli space of branch divisors.  It is therefore a natural question whether one could use this cycle to extract information about intersection theory on some classical moduli space. Further corroborating this question is the analogous case of simple Hurwitz numbers, counting covers of $\PP^1$ with one point of specified special ramification. Here the celebrated ELSV formula \cite{ELSV2}, which expresses these numbers as  tautological intersection numbers on the moduli spaces of curves, has at the same time explained the combinatorial properties of the Hurwitz numbers and provided a wealth of remarkable geometric consequences. One among all, the ELSV formula is a key ingredient to Okounkov and Pandharipande's proof of Witten's conjecture \cite{op:witten}.
The fact that double Hurwitz numbers share simlar combinatorial properties has led Goulden, Jackson and Vakil to conjecture the existence of an {\it ELSV formula for double Hurwitz numbers}, in the form of an intersection of tautological classes on some (family of) compactification(s) of the Picard Stack (\cite{gjv:dhn}, Conjecture 3.5) . To this day such a formula has not been found.

In this article we take a different approach, and express the double Hurwitz number as an intersection of tautological classes on $\cms_{g,n}$ ({\bf Proposition \ref{prop:formula}}). We interpret the double Hurwitz zero dimensional cycle essentially as the top intersection of $x_i\tilde\psi_i$, a psi class on the moduli space of rubber relative stable maps to $\PP^1$. We then pushforward this expression to the moduli space of curves via the stabilization morphism, obtaining a formula in terms of tautological intersections on $\cms_{g,n}$. There are a few fundamental facts that combine to show that our formula  explains the chamber structure and piecewise polynomiality of  double Hurwitz numbers ({\bf Corollary \ref{cor:poly}}). First, psi classes on spaces of relative stable maps equal the pullback of psi classes on the moduli spaces of curves plus some chamber dependent boundary corrections. Therefore, projection formula gives us the double Hurwitz number as intersection numbers of psi classes and cycles $c_\Delta$ supported on chamber-dependent boundary strata $\Delta$. For a given $\Delta$ in $\cms_{g,n}$ such that the dual graph has genus $l$, the cycle $c_\Delta$ is obtained as follows: there is a finite collection of boundary strata in the space of maps whose pushforward is supported on $\Delta$. Such a collection is parameterized by the lattice points of an $l$-dimensional polytope 
$P_\Delta$ whose faces are given by homogeneous linear equations in the entries of $\bx$. From each such boundary stratum we pushforward the restriction of the virtual fundamental class $[\relmaps{g}{\bx}]^{vir}$.  When doing so, ghost automorphisms in the  boundary of the space of relative stable maps give a contribution for each node of $\Delta$ that is a linear factor in the coordinates of $\bx$ and of the ambient lattice for $P_\Delta$. If we assume that the pushforward of the virtual class of the moduli space of rubber relative stable maps  of  a given genus $g$ is a polynomial class of degree $2g$, some elementary bookkeeping shows that the coefficient of $\Delta$ is a polynomial class of degree $4g-3+n$.

The polynomiality of the pushforward of the virtual class of relative stable maps is not just wishful thinking. In recent years such class has been studied intensely by several groups of mathematicians from different areas, perhaps due to Eliashberg's request for a working understanding of  the {\it double ramification cycle} $\HH{g}{\bx}:=\stab_\ast[\relmaps{g}{\bx}]^{vir}$, which should play a key role in symplectic field theory \cite{egh:sft}. In \cite{H11}, Hain showed that the restriction of $\HH{g}{\bx}$ to $\sM_{g,n}^{ct}$ (curves of compact type) is a homogeneous polynomial of degree $2g$ (sometimes called the {\it Hain class}). Grushevski and Zakharov (\cite{GZ1,GZ2}) extend further this class to a larger partial compactification of $\sM_{g,n}$, including curves with one loop in the dual graph: the class is still a polynomial of degree $2g$, but it is no longer homogeneous. Finally Buryak-Shadrin-Spitz-Zvonkine \cite{bssz:psi} study the intersections of $\HH{g}{\bx}$ with monomials in psi classes, obtaining results that are consistent with the polynomiality assumption. These authors also use in an essential way the comparison of psi classes among  the various moduli spaces in question. We remark that polynomiality of the double ramification cycle holds trivially in genus $0$ and is true in genus $1$. It is also interesting that the correction between the Hain class and the double ramification cycle is very suggestively just the class $\lambda_1$, with constant coefficient $1$.

Analyzing the chamber dependence of our formula for double Hurwitz numbers, we establish an intersection theoretic formula for the wall crossings ({\bf{Theorem \ref{thm:wc}}}).  We show it is controlled by the pushforward of a specific divisor in the moduli space of relative stable maps, which we call the wall-crossing divisor. In the rational case, a considerably simpler geometric wall-crossing formula is expressed ({\bf Proposition \ref{prop:gz}}) in terms of a unique irreducible divisor $D$ (Figure \ref{fig:genus0}).

The paper is organized as follows. Section~\ref{sec:setup} contains basic definitions and background information needed to build up diagram \eqref{diag:central}, which plays a key role in this work.  In Section~\ref{sec:vanishing} we provide technical vanishing lemmas for the pushforwards of boundary loci in $\relrub{g}{\bx}$ through the stabilization morphism.  
In Section~\ref{sec:psi} we compare pullbacks and pushforwards of the various cotangent line bundle classes along these maps. Most of the results in Section~\ref{sec:psi} are already found in \cite{ion:tc}, \cite{bssz:psi}, but we present some streamlined proofs in algebro-geometric language. In section \ref{sec:WC} we bring everything together to provide the geometric realization of double Hurwitz numbers and their wall crossings. 

\subsection{Acknowledgements}
We are especially grateful to Arend Bayer, who actively collaborated with the first author at the initial stages of the project. Many helpful discussions with other mathematicians, including Brian Ossermann, Sergey Shadrin, Jonathan Wise, Dmitri Zakharov and  Dimitri Zvonkine have helped shaping this work into its current form. We are grateful to AIM that brought us all together in the February 2012 workshop on integrable systems in Gromov-Witten and symplectic field theory to exchange ideas and points of view on the double ramification cycle.

\section{Setup}
\label{sec:setup}
In this section we quickly recall the necessary background in double Hurwitz theory and we explain the objects involved in the central diagram \eqref{diag:central} inspiring Theorem~\ref{cor:poly}. 

\subsection{Double Hurwitz numbers} 
\label{sec:dhn}
Fix $n\in\Z_{>0}$ a positive integer and let \begin{equation}\displaystyle{\bx}\in\Z^n_0=\left\{\bx\in\Z^n|\sum_ix_i=0\right\}\end{equation} be an $n$-tuple of integers summing to zero.  Denote by $\bx_{0}$ and $\bx_{\infty}$ the tuples given by the positive and negative components of $\bx$ respectively.  Double Hurwitz numbers $H_g(\bx)$ are invariants giving a (automorphism-weighted) count of the genus $g$ covers of $\PP^1$ with ramification profiles over $0$ and $\infty$ prescribed by $\bx_0$ and $\bx_\infty$.  These numbers determine a function 
\begin{equation}
H_g:\Z^n_0\longrightarrow\Q.  
\end{equation}
on the integral lattice of zero-sum $n$-tuple of integers.
Goulden, Jackson, and Vakil \cite[Theorem~2.1]{gjv:dhn}, show that $H_g(\bx)$ is piecewise polynomial of degree $4g-3+n$.  A complete combinatorial description of this piecewise polynomiality behavior is given by Shadrin--Shapiro--Vainshtein \cite{ssv:gz} in genus 0 and by Cavalieri--Johnson--Markwig\cite{cjm:wc} in full genera.  For $I$ ranging among the proper subsets of $\{1,\ldots,n\}$, the hyperplanes 
\begin{equation}W_I=\left\{\sum_{i\in I}x_i=0\right\}\end{equation}
form walls in $\R^n$ defining the  chambers of polynomiality $\mathfrak{c}$ as the connected components of the complement. Explicit wall-crossing formulae in \cite{ssv:gz}, \cite{cjm:wc}, are modular and inductive, in the sense that they describe the variation of a Hurwitz polynomial across a wall in terms of products of double Hurwitz numbers with smaller invariants.

\subsection{Curves}
The Deligne-Mumford compactification $\cms_{g,n}$ of the moduli space of curves parameterizes families of stable projective genus $g$ algebraic curves with $n$ marked points and has dimension $3g-3+n$.  Stability is defined to be ampleness of the log canonical divisor $\omega_C+\sum_ip_i$.  Much of the intersection theory of this moduli space is captured by the tautological ring $\sR^\ast(\cms_{g,n})\subset \ch^\ast(\cms_{g,n})$, a naturally defined subring containing most of the known geometrically defined Chow classes.  These include the cotangent line bundle classes
\begin{equation}\psi_i \in \sR^1(\cms_{g,n}),\, i=1,\ldots, n\end{equation}
the Chern classes of the Hodge bundle
\begin{equation}\lambda_j\in \sR^j(\cms_{g,n}),,\, j=1,\ldots,g\end{equation}
and the various boundary strata.

\subsection{Losev-Manin spaces}One may alter the stability condition by assigning weights $\left\{a_i\right\}_{i=1}^n$ to the marked points and requiring $\omega_C+\sum_ia_ip_i$ to be ample.  This produces a moduli space $\cms_{g,n}(a_1,\ldots,a_n)$ of weighted stable curves (\cite{h:wsc}) with various new combinatorial properties.  In genus $0$, assigninging  weight one to two points denoted $0$ and $\infty$  and giving an infinitesimally small weight $\epsilon$ to the remaining $r$ others results in such a space, also known as Losev-Manin space (\cite{lm:lms}).  We work with its quotient by the symmetric group action forgetting the ordering of the $r$ ``shadow'' points:
\begin{equation}\cms_{\br}:=\left[\cms_{0,2+r}(1,1,\epsilon,\ldots,\epsilon)/\mathcal{S}_r\right].\end{equation}  There is a natural contraction morphism $c: \cms_{0,n}\to \cms_{0,2+r}(1,1,\epsilon,\ldots,\epsilon)$   and the cotangent line bundle classes $\widehat{\psi}_0$ and $\widehat{\psi}_\infty$ at the fully weighted special points pullback via $c$ to the corresponding ordinary $\psi$ classes on $\cms_{0,n}$.

\subsection{Relative maps to $\PP^1$}
The moduli space \begin{equation}\relmaps{g}{\bx}:=\relmaps{g}{\bx_0[0],\bx_\infty[\infty]}\end{equation} of relative stable maps (\cite{GV05}) to $\PP^1$ parameterizes degree $d$ stable maps relative to the points $0$ and $\infty$ with prescribed ramification given by $\bx_0$ and $\bx_\infty$ respectively.  We consider the variant of this space $\relrub{g}{\bx}$ in which the target is an unparameterized or ``rubber" $\PP^1$.  Closed points in this space are branched degree $d$ maps $f:C\to T$ to a semi-stable chain $T$ of projective lines with the appropriate ramification over the two special points, which lie on the external components of $T$.  The pre-images of the relative divisors $0$ and $\infty$ are considered marked. For each such marked point $i$ we have a cotangent line bundle class that we denote $\widetilde{\psi}_i$.
This moduli space admits a virtual fundamental class of  dimension $2g-3+n$.



\subsection{The central diagram}
The relative stable maps space $\relrub{g}{\bx}$ admits a natural stabilization morphism 
\begin{equation}
\stab:\relrub{g}{\bx}\to\cms_{g,n}
\end{equation}
defined by sending a relative stable map $f:C\to T$ to the Deligne-Mumford stabilization $\oC$ of the source curve.  In \cite[Theorem~1]{FP05}, Faber and Pandharipande show that the pushforward 
\begin{equation}
\HH{g}{\bx}:=\stab_\ast\left[\relrub{g}{\bx}\right]^{\rm vir}
\end{equation}
is a tautological class of codimension $g$. In \cite{GV03} this is called the \emph{double Hurwitz class}. More recently the name {\it double ramification cycle} has been adopted in  \cite{GZ1,GZ2} and \cite{bssz:psi} for the cycle naturally representing $\HH{g}{\bx}$.

The space $\relrub{g}{\bx}$ also admits a natural branch morphism 
\begin{equation}
\br:\relrub{g}{\bx}\to\cms_{\br}
\end{equation}
defined by sending a map $f:C\to T$ to the expanded target $T$ with the branch divisor away from $0$ and $\infty$ marked by the $r=2g-2+n$ lightly weighted markings.  The double Hurwitz number $H_g(\bx)$ is the degree of the branch morphism:
\begin{equation}
H_g(x)= \int_{[\relrub{g}{\bx}]^{vir}}br^\ast([pt]).
\end{equation}

We use the diagram
\begin{equation}\label{diag:central}
\xymatrix{\relrub{g}{\bx} \ar[r]^{stab} \ar[d]^{br} & \cms_{g,n} \\
\cms_{br}
}
\end{equation}
to espress $H_g(\bx)$ as an intersection number on $\cms_{g,n}$ as:
\beq
\label{eq:dhnif}
H_g(\bx)[pt.]= \stab_\ast(\br^\ast([pt.])).
\eeq


\section{Vanishing of boundary loci}\label{sec:vanishing}

We now provide a statement about the vanishing of boundary loci in $\relrub{g}{\bx}$.  Let $\tilde{\Delta}_{g,k}$ denote an irreducible codimension $k$ boundary stratum in $\relrub{g}{\bx}$ described generically by a relative stable map to a $k$-th expansion $T$ of $\PP^1$ (see Figure~\ref{fig1}), and denote $\tilde{\Delta}_{g,k}^{vir}$ the class obtained by capping it with the virtual fundamental class of the moduli space of maps.  

\begin{figure}[h]
\begin{center}
		\includegraphics[width=\textwidth]{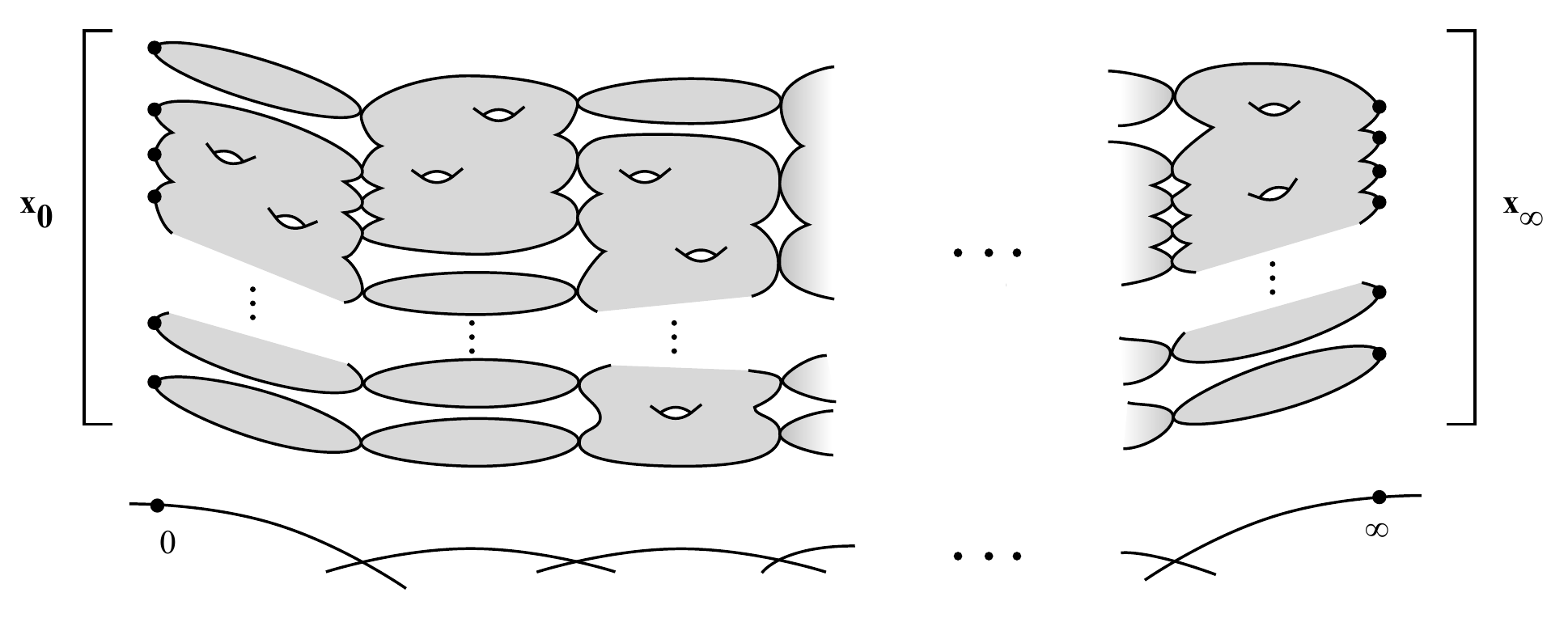}
\caption{The generic relative stable map of an irreducible codimension $k$ boundary stratum $\tilde{\Delta}_{g,k}$ of $\relrub{g}{\bx}$.  The target chain of rational curves has $k+1$ components.}
\label{fig1}
\end{center}
\end{figure}


To begin, we introduce notation that allows us to discuss the parts of Figure~\ref{fig1} in a meaningful way.  For a fixed locus $\tilde{\Delta}_{g,k}$, denote by $(f_\eta:C_\eta \to T)\in\tilde{\Delta}_{g,k}$ the generic relative stable map characterizing $\tilde{\Delta}_{g,k}$, and let $\pi:C_\eta\to\oC_\eta$ be its Deligne-Mumford stabilization.  Recall, $C_\eta$ is a nodal pre-stable curve, $\oC_\eta:=\stab(f_\eta)$ is a nodal stable curve, and $T\in\sM^\sst$ is an unparameterized chain of $k$ projective lines.  

Call an irreducible component $C'\in C_\eta$ \emph{trivial} if it is contracted by $\pi$.  Non-contracted irreducible components are called \emph{non-trivial}.  The trivial components are precisely those that map via $f_\eta$ as a Galois cover to a single component of $T$, fully ramified over 2 branch points. Similarly, for each node $\op\in\oC_\eta$ in the stabilization, we call the pre-images $\pi^{-1}(\op)$ \emph{non-trivial} nodes of $C_\eta$.  Notice that a non-trivial node can either be an isolated node or a chain of trivial components that stabilize to a node.  A node $p\in C_\eta$ is called \emph{trivial} if it is not a non-trivial node (see Figure~\ref{fig2} for a labeling of such parts).  Denote by $\gamma$ the number of non-trivial components of $C_\eta$ and $\delta$ the number of non-trivial nodes.  We are also concerned with the number $l$ of loops in $C_\eta$, that is, the number of minimal cycles in its dual graph.

\begin{figure}[h]
\begin{center}
		\includegraphics[width=\textwidth]{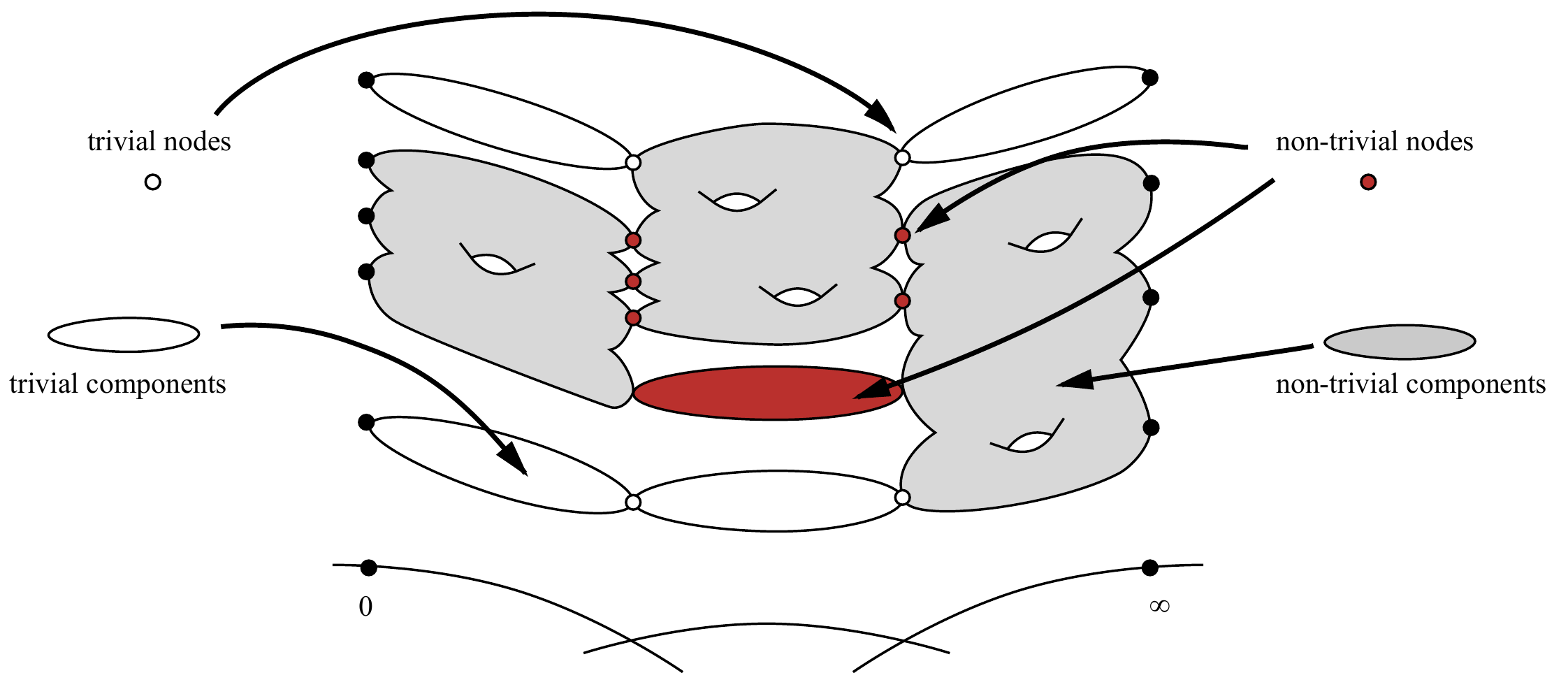}
\caption{A diagram labeling trivial and non--trivial nodes and components.}
\label{fig2}
\end{center}
\end{figure}

In this section we give a combinatorial conditions for the vanishing of the pushforward $\stab_\ast\tilde{\Delta}_{g,k}^{vir}$. We warm up recalling the statement in genus $0$.


\begin{lemma}[\cite{BCM}]
Let $\tilde{\Delta}_{0,k}\subset\relrub{0}{\bx}$ be a boundary stratum described generically as having an expanded target $T$ with $k$ nodes, as in Figure~\ref{fig1}.  Then $\stab_\ast\left[\tilde{\Delta}_{0,k}\right]\neq 0$ if and only if $\delta=k$.
\end{lemma}

\begin{proof}
In this case ${\rm dim}\left(\relrub{0}{\bx}\right)=\dim\left(\cms_{0,n}\right)=n-3$, thus $\stab_\ast\left[\tilde{\Delta}_{0,k}\right]\neq0$ if and only if $\stab(\tilde{\Delta}_{0,k})$ is a locus of codimension $k$.  By the well known combinatorial description of the boundary strata of $\cms_{0,n}$, the stabilization $\oC_\eta$ must have precicely $k$ nodes, i.e. $\delta=k$.  Note that, since a genus 0 pre-nodal curve is compact type we must have $l=0$ and $\gamma=k+1$ non-trivial components, one for each irreducible component in the expanded target $T$.  \end{proof}

This generalizes to the following statement in higher genus.

\begin{lemma}
\label{lem:bound}
Let $\tilde{\Delta}_{g,k}\subset\relrub{g}{\bx}$  be a boundary stratum described generically as having an expanded target $T$ with $k$ components, as in Figure~\ref{fig1}.  Then $\stab_\ast\left[\tilde{\Delta}_{g,k}^{vir}\right]\neq 0$ if and only if $\delta=l+k$.
\end{lemma}

\begin{proof}

For each $i=1,\ldots,\gamma$ let  $C_i$ be the collection of non-trivial components of $C_\eta$, $g_i$ their respective genus, and $T_i$ the irreducible component of $T$ onto which $C_i$ surjects via $f_\eta$.  For each $i$, the map $f_\eta$ restricted to $C_i$ describes the generic point of a  moduli space $\relrub{g_i}{\bx_i}$ of relative maps to $T_i\cong \PP^1$ where $\bx_i$ is the vector of integers recording the ramification profile of this restriction over the two special points of $T_i$.  The locus $\tilde{\Delta}_{g,k}$ maps to a fiber product of moduli spaces
\beq
\label{push}
\xymatrix{
\tilde{\Delta}_{g,k}\ar[r]_p \ar[d]_{stab}  &  \mathcal{M}:= \ar[dl]^{\prod stab_i}& \hspace{-1cm}\relrub{g_1}{\bx_1}\times_{\sM_1}\relrub{g_2}{\bx_2}\times_{\sM_2}\cdots\times_{\sM_{\gamma-1}}\relrub{g_{\gamma}}{\bx_\gamma}\\
 \cms_{g,n} &  & 
}
\eeq

\noindent where the $\sM_i$ are zero dimensional moduli spaces of maps of non-trivial nodes along which our maps glue to form the objects of $\relrub{g}{\bx}$. Diagram \eqref{push} is natural with respect to virtual classes, hence we have that $\stab_\ast(\tilde{\Delta}_{g,k}^{vir})$ factors through the map $p$. It is immediate to see that $p$ has positive dimensional fibers unless there is precisely one non-trivial component over each component of the expanded target. A necessary condition for $\stab_\ast(\tilde{\Delta}_{g,k\tilde{\Delta}}^{vir})\not= 0$ is then that $\gamma=k+1$.
Further, we have that $l= \delta -(\gamma-1)$, which immediately implies $\delta= l+k$.

Note that the dimension of $\stab(\tilde{\Delta}_{g,k}^{vir})$ is
\begin{equation*}
\sum_{i=1}^{k+1} \left(2g_i-3+l(\bx_i)\right)=2(g-l)-3\gamma+\sum_{i=1}^\gamma l(\bx_i)= 2g-3+n,
\end{equation*}
where the last equality follows from
\begin{equation}
\sum_{i=1}^\gamma l(\bx_i)=n+2\delta
\end{equation}
since the parts of the tuples $\bx_i$ corresponding to non-trivial nodes are counted twice, once for each of the two non-trivial components branching at that node, and the rest correspond exactly to the original ramification data $\bx$.  

To conclude our proof it suffices to observe that $\stab_\ast(\tilde{\Delta}_{g,k}^{vir})$ is a double Hurwitz class on each of the factors of the strata that support it, and is therefore non-zero as the double Hurwitz classes aren't.

\end{proof}

\section{Comparisons lemmas}\label{sec:psi}
This section contains the technical lemmas needed to compare and evaluate the various $\psi$ classes appearing in diagram \eqref{diag:central}.

\begin{lemma}\label{lem:ev}
Denote by $\hat{\psi_0}\in \sR^1(\cms_{\br})$ the $\psi$ class corresponding to one of the two points with weight one. Then
\beq\widehat{\psi}_0^{2g-3+n}=\dfrac{1}{r!}[\text{pt.}]\eeq
\end{lemma}
\begin{proof}
Consider the contraction morphism 
\begin{equation}
c: \cms_{0,r+2}\to \overline{\sM}_0(1,1,\varepsilon,\ldots, \varepsilon).
\end{equation}
Since the point $0$ has weight 1,  we have that $c^\ast(\widehat{\psi}_0)=\psi_1$. Therefore the top intersection of the $\psi$ class on the weighted curves space is equal to the top intersection of a  $\psi$ class on an (ordinary) $\cms_{0,n}$, which is $1[pt.]$. 
The $1/r!$ factor comes from the fact that the branch space we consider is a $S_r$ quotient of $\overline{\sM}_0(1,1,\varepsilon,\ldots, \varepsilon)$.

\end{proof}

The following Lemma is just an adaptation to our notation and context of \cite[Lemma~1.17]{ion:tc}.

\begin{lemma}\label{lem:ion}
Let $\widetilde\psi_i \in \sR^1( \relrub{g}{\bx})$ correspond to a ramification point of order $x_i$ mapping to $0$. Then: \begin{equation}br^\ast(\widehat{\psi}_0)= x_i\widetilde{\psi_i}\end{equation} 
\end{lemma}
\begin{proof} Consider the commutative diagram
\begin{equation}
\xymatrix{
\relrubU{g}{\bx}\ar[dr] \ar[d]^{f}&\\
\overline{\mathcal{U}}_{br} \ar[dr]  & \relrub{g}{\bx} \ar[d]^{br} \ar@/_/[ul]_{s_i} \\
  & \cms_{br} \ar@/^/[ul]^{s_0}
}
\end{equation}
where $\relrubU{g}{x}$ and $\overline{\sU}_{br}$ are the respective universal curves, and the maps $s_i$ and $s_0$ are the sections for the respective marked points.  We then have the following chain of equalities:
\begin{equation}
br^\ast(\widehat{\psi}_0)= -br^\ast s_0^\ast(0)= -s_i^\ast f^\ast(0)= -s_i^\ast(x_is_i)= x_i\widetilde{\psi_i}.
\end{equation}
\end{proof}

\begin{lemma}\label{lem:cor}
Let $\widetilde\psi_i \in \sR^1( \relrub{g}{\bx})$ correspond to a ramification point of order $x_i$, and $\psi_i\in \sR^1( \cms_{g,n})$  be the psi class for the same mark after forgetting the map.  Denote by $\sD_i$  the divisor parameterizing maps where the $i$-th mark is supported on a trivial component. Then
\begin{equation}\widetilde{\psi_i} = stab^\ast\psi_i+ \frac{1}{x_i}\sD_i.\end{equation}
\end{lemma}
\begin{proof} Consider the commutative diagram
\begin{equation}
\xymatrix{
\relrubU{g}{\bx}\ar[r]^{Stab} \ar[d]^{}&\overline{\mathcal{U}}_{g,n}\ar[d]  \\
 \relrub{g}{\bx}\ar@/_/[u]_{\tilde{s_i}}  \ar[r]^{stab}  & \cms_{g,n} \ar@/_/[u]_{s_i}
 }
\end{equation}
where $\relrubU{g}{x}$ and $\overline{\sU}_{g,n}$ are the respective universal curves, and the maps $\tilde{s_i}$ and $s_i$ are the sections for the respective marked points.  Denote by $E_i\subset \relrubU{g}{x}$ the locus of the contracted rational bubbles supporting the $i$-th mark. We observe that $E_i$ intersects the image of the section $\tilde{s_i}$ on the locus $Z_i$ parameterizing the $x_i$-twisted nodes where the trivial components attach to the rest of the curve; $Z_i$ is a $\mathbb{Z}_{x_i}$ gerbe over $\sD_i$. Abusing notation and denoting by $s_i$ (resp. $\tilde{s_i}$) both a section and its image:
\begin{equation}
\tilde{\psi_i}= -\tilde{s_i}^\ast(\tilde{s_i})= -\tilde{s_i}^\ast(Stab^\ast(s_i)-E_i )= stab^\ast (\psi_i) + \frac{1}{x_i}\sD_i.
\end{equation}

\end{proof}

\begin{figure}[tb]\begin{center}
		\includegraphics[width=0.6\textwidth]{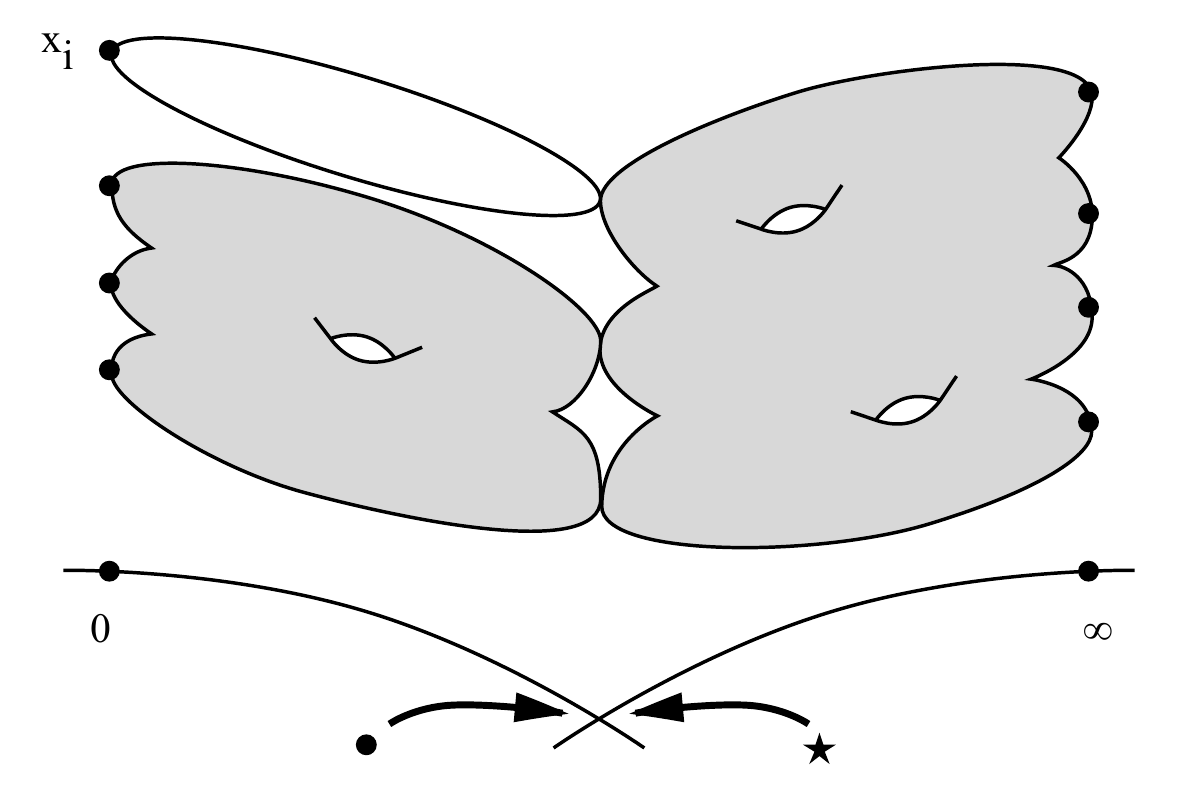}
\end{center}
\caption{The map parameterized by a generic point of an irreducible boundary divisor $D_i\subset \relrub{g}{\bx}$, and the corresponding expanded target.  The  corresponding divisor $T_i$ in the Losev-Manin space is naturally a  product  of two other Losev-Manin spaces via the morphism gluing the weight $1$ points $\bullet$ and $\star$; via pullback from the natural projections we define the classes $\widehat{\psi}_\bullet, \widehat{\psi}_\star$, supported on $T_i$, and corresponding to the psi classes ``at the shadows of the node".  }
\label{fig:psistar}
\end{figure}

\begin{lemma}\label{lem:side}
With all notation as in  Lemma \ref{lem:cor}, let $D_i$ be an irreducible component of $\sD_i$ and $\widehat{\psi}_\bullet, \widehat{\psi}_\star$ the classes introduced and explained in Figure~\ref{fig:psistar}. Then:
\beq
\label{eq:side}
\tilde{\psi_i}_{|D}= -\frac{1}{x_i}br^\ast(\widehat{\psi}_\bullet)
\eeq
\begin{proof}

 Consider the restriction of the exceptional locus ${E_i}_{|\pi^\ast(D)}\to D$. It is a projective bundle with fibers $\PP^1$, and we can identify the locus of nodes $Z_i$ with the $0$ section, and the image of $\tilde{s_i}$ with the infinity section of this bundle. Denoting $N_{0}$ and $N_{\infty}$ the respective normal bundles, it is a fact that $c_1(N_0)=-c_1(N_{\infty})$.
Then formula \eqref{eq:side} follows immediately from Lemma \ref{lem:ion}.

\end{proof}
\end{lemma}

\section{Geometric wall crossing} \label{sec:WC}
Fix a genus $g$ and ramification data $\bx$ in a chamber of polynomiality $\mathfrak{c}$. Our first result is a formula computing the double Hurwitz number (i.e. $0$-dimensional cycle)  $H_g(\bx)$ as an intersection number on the moduli space of curves.

\begin{prop}\label{prop:formula} The following equality of zero dimensional cycles in $\cms_{g,n}$ holds:
\begin{equation}\label{eq:formula}H_g(\bx)[\pt]=r! \sum_{k=0}^{2g-3+n}\binom{2g-3+n}{k}x_i^{2g-3+n-k}\psi_i^{2g-3+n-k} \stab_\ast\left( \left(\sD_i^{\fc}\right)^{k}\cap\left[\relrub{g}{\bx}\right]^{\rm vir}\right).\end{equation}
\end{prop}
\begin{proof}
 Recall diagram  \eqref{diag:central} and  formula \eqref{eq:dhnif}. 
An application of Lemmas~\ref{lem:ev}, \ref{lem:ion} and \ref{lem:cor}   provides the chain of equalities:
\begin{align*}
\br^\ast[\pt]\cap\left[\relrub{g}{\bx}\right]^{\rm vir} \stackrel{\ref{lem:ev}}{=}& r!\br^\ast\left(\widehat{\psi}^{2g-3+n}\right)\cap\left[\relrub{g}{\bx}\right]^{\rm vir}\\
\stackrel{\ref{lem:ion}}{=}& r!\left(x_i\widetilde{\psi}_i\right)^{2g-3+n}\cap\left[\relrub{g}{\bx}\right]^{\rm vir}\\
\stackrel{\ref{lem:cor}}{=}& r!\left(x_i\stab^\ast\psi_i+\sD_i^{\fc}\right)^{2g-3+n}\cap\left[\relrub{g}{\bx}\right]^{\rm vir}
\end{align*}
in the chow ring of $\relrub{g}{\bx}$ for any $i$-th marked pre-image of $0$.  Pushing forward via the stabilization morphism, expanding, and applying the projection formula yields:
\begin{align}
H_g(\bx)[\pt] =& r! \stab_\ast\left(\left(x_i\stab^\ast\psi_i+\sD_i^{\fc}\right)^{2g-3+n}\cap\left[\relrub{g}{\bx}\right]^{\rm vir}\right) \nonumber\\
=& r! \stab_\ast\left( \sum_{k=0}^{2g-3+n}\binom{2g-3+n}{k}(x_i\stab^\ast\psi_i)^{2g-3+n-k}\left(\sD_i^{\fc}\right)^{k}\cap\left[\relrub{g}{\bx}\right]^{\rm vir}\right)\nonumber\\
=& r! \sum_{k=0}^{2g-3+n}\binom{2g-3+n}{k}(x_i\psi_i)^{2g-3+n-k}\stab_\ast\left( \left(\sD_i^{\fc}\right)^{k}\cap\left[\relrub{g}{\bx}\right]^{\rm vir}\right). \label{nice}
\end{align}
\end{proof}

We regard the following corollary as evidence for polynomiality of the double Hurwitz class. 

\begin{cor}\label{cor:poly}
Assuming $\HH{g}{\bx}$ is a Chow valued  (possibly piecewise) polynomial of degree $2g$, $H_g(\bx)$ is a piecewise polynomial function of degree $4g-3+n$.
\end{cor}
\begin{proof}
We show that for every $k$,  $\stab_\ast\left( \left(\sD_i^{\fc}\right)^{k}\cap\left[\relrub{g}{\bx}\right]^{\rm vir}\right)$ in chamber $\fc$ is a polynomial class of degree $2g+k$.  Consider an irreducible boundary stratum $\Delta_{g,k}$ of codimension $k$ in $\cms_{g,n}$. With the same notation of Section \ref{sec:vanishing}, $\delta$ denotes the number of nodes and $l$ the combinatorial genus of the dual graph. $\Delta_{g,k}$ is the (possibly empty) image of a collection of irreducible boundary strata $\tilde{\Delta}_{g,k, m_1, \ldots, m_l}$, indexed by an $l-tuple$ of integers parameterized by the lattice points of a polytope whose faces are given by linear functions of the $x_i$'s.
Each irreducible component $\tilde{\Delta}_{g,k, m_1, \ldots, m_l}$ pushes forward to a class which is a polynomial of degree $2\sum g_i +\delta$ in the $x$'s and in the $m$'s, considered as variables. Adding all such contributions over the lattice points of the constraining polytope shows that the coefficient of $\Delta_{g,k}$ is a polynomial of degree $2\sum g_i +\delta+l$ in the $x$ variables. Now we invoke the relation $ \delta=l+k $ from Lemma \ref{lem:bound},  and recall that $g=\sum g_i +l$, to conclude the proof.
\end{proof}

As a consequence of formula \eqref{nice} we obtain an intersection theoretic formula for the wall crossings. We briefly recall (see \cite{ssv:gz,cjm:wc} for a more thorough definition) that by wall crossing we mean the difference of the double Hurwitz polynomials corresponding to two adjacent chambers.

\begin{thm}\label{thm:wc}
 Consider two adjacent chambers $\fc_1$ and $\fc_2$, separated by a wall \[W_{I}= \left\{\sum_{i\in I} x_i=0\right\}.\] Assume $\bx\in \fc_2 (\mbox{where} \sum_{i\in I}x_i>0)$ and denote by $\sD_i^{W_I}= \sD_i^{\fc_2}-\sD_i^{\fc_1} \footnote{$\sD_i^{\fc_1}$ is a divisor defined in moduli spaces corresponding to $\bx\in \fc_1$. We define it for $\bx\in \fc_2$ by considering only the irreducible components that exist when evaluated at $\bx$. Specifically this means that all ghost automorphism weights of non-trivial nodes remain positive.} \in \relrub{g}{\bx}$, the wall crossing divisor. Then the wall crossing formula $WC_I(\bx)$ is a polynomial whose coefficients are classes supported on the pushforward of $\sD_i^{W_I}$.
 \end{thm}
 
\begin{proof}
By applying formula \eqref{nice}, we have
\begin{align}\hspace{-1cm}
WC_I(\bx)=& r! \stab_\ast\left(\left(\left(x_i\stab^\ast\psi_i+\sD_i^{\fc_2}\right)^{2g-3+n}-\left( x_i\stab^\ast\psi_i+\sD_i^{\fc_1}\right)^{2g-3+n}\right)\cap\left[\relrub{g}{\bx}\right]^{\rm vir}\right) \nonumber\\
=& r! \stab_\ast\left( \sum_{k=0}^{2g-3+n}\binom{2g-3+n}{k}(x_i\stab^\ast\psi_i)^{2g-3+n-k}\left(\left(\sD_i^{\fc_2}\right)^{k}-\left(\sD_i^{\fc_1}\right)^{k}\right)\cap\left[\relrub{g}{\bx}\right]^{\rm vir}\right)\nonumber\\
=& r! \sum_{k=0}^{2g-3+n}\sum_{m=0}^{k-1}\binom{2g-3+n}{k}\binom{k}{m}(x_i\psi_i)^{2g-3+n-k}\stab_\ast\left( \left(\sD_i^{\fc_1}\right)^{m} \left(\sD_i^{W_I}\right)^{k-m}\cap\left[\relrub{g}{\bx}\right]^{\rm vir}\right). \label{for:wc} \nonumber\\
 & 
\end{align}
Since the second summation terminates at $k-1$, formula \eqref{for:wc} is  supported on $\stab_\ast(\sD_i^{W_I})$.
\end{proof}

\subsection*{An example in genus $0$.}  
In this example we exhibit a wall crossing in the polynomiality of genus $0$ five pointed double Hurwitz numbers using the above formulas.  Our path begins in the \emph{totally negative chamber} where $x_1$ and $x_2$ are positive; $x_3,x_4,$ and $x_5$ are negative; $x_1$ is very large; and $x_2$ is very small:
\begin{align*}
x_3,x_4,x_5<-\epsilon&\\
x_1>> 0&\\
\epsilon>x_2>0&.
\end{align*}
Denote this chamber by $\fc_1$.  We will cross the wall $W_{\left\{x_2,x_5\right\}}$ given by $x_2+x_5=0$, i.e. moving from $\fc_1$ where $x_2+x_5<0$ to the adjacent chamber $\fc_2$ where $x_2+x_5>0$, allowing $|x_2|$ to overtake $|x_5|$.  As described in \cite{ssv:gz}, the double Hurwitz polynomials take the form:
\begin{description}
	\item[$\fc_1$] $$
H_0(\bx)= 3! x_1^2
$$
	\item[$\fc_2$]$$
H_0(\bx)= 3!x_1(x_1+x_2+x_5) 
$$
\end{description}
with the wall crossing given by the polynomial contribution \[WC(\bx)_{\left\{x_2,x_5\right\}}=3(x_2+x_5) 2x_1.\]

\begin{figure}[tb]
\begin{center}
		\includegraphics[width=0.6\textwidth]{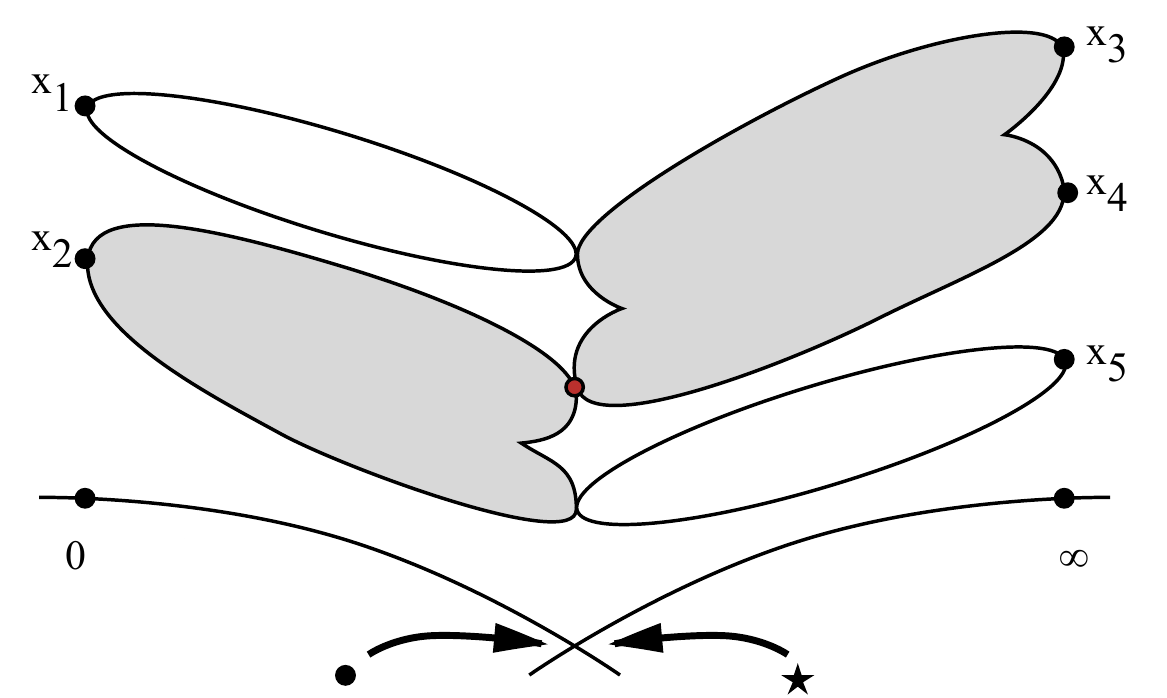}
\end{center}
\caption{A generic map in the chamber 2 divisor  $\sD_1^{\fc_2}$.}  
\label{fig:example}
\end{figure}

In the chamber $\fc_1$, the divisor $\sD_1^{\fc_1}$ is empty, so Formula~\eqref{eq:formula} reduces easily to $H_0(\bx)= 3! x_1^2$ since only the $k=0$ term in the sum contributes.  Crossing the wall $W_{\left\{x_2,x_5\right\}}$ to $\fc_2$, the divisor $\sD_1^{\fc_2}$ is non--empty and takes the form of Figure~\ref{fig:example}.  Notice that, after an inconsequential replacement of $I$ with $I^c$, $\sD_1^{\fc_1}$ is an example of the more general form of a genus $0$ wall crossing divisor depicted in Figure~\ref{fig:genus0} below.  This is because in $\fc_1$ the divisor is empty so the $\fc_2$ divisor alone controls the wall--crossing.  

Formula~\ref{eq:formula} determines our double Hurwitz polynomial in $\fc_2$ to be the intersection
\begin{equation}\label{eq:example}
H_0(\bx)= 3!(x_1^2 + 2x_1\psi_1\stab_\ast \sD_{\rm vir} +\stab_\ast(\sD_{\rm vir}^2))
\end{equation}
where $\sD_{\rm vir}:=\sD_1^{\fc_2}\cap\left[\relrub{0}{\bx}\right]^{\rm vir}$.  We are left with the task of computing the two intersections $\psi_1\stab_\ast (\sD_{\rm vir})$ and $\stab_\ast(\sD_{\rm vir}^2)$ on $\cms_{0.n}$.

For the first, note that $\stab_\ast(\sD_{\rm vir})= (x_2+x_5) D(1,3,4|2,5)$.  The polynomial factor is provided by the multiplicity at the non--trivial node.   Evaluating the psi class results in a contribution of $(x_2+x_5)[pt.]$.  

The second requires computing the self--intersection of $\sD_i^{\fc_2}$.  The normal bundle of $\sD_i^{\fc_2}$ contributes $-\br^\ast\widehat{\psi}_\bullet-\br^\ast\widehat{\psi}_\star$ to the excess intersection, but only $-\br^\ast\widehat{\psi}_\star$ remains since $\widehat{\psi}_\bullet$ is supported on a rational curve with three special points ($0$, the node, and the single branch point on this component of the expanded target).  As in the proof of Lemma~\ref{lem:bound}, our divisor maps to a fiber product of relative stable map spaces determined by the non--trivial components in Figure~\ref{fig:example}.  This is natural with respect to virtual classes.  The resulting intersection becomes 
\begin{align*}
\stab_\ast(-\br^\ast(\widehat{\psi}_\star)\sD_{\rm vir}){=}& \stab_\ast\left(-(x_2+x_5)\br^\ast(\widehat{\psi}_\star)\cap\left[\relrub{0}{x_1,x_2+x_5,x_3,x_4)}\right]^{\rm vir}\right)\\
\stackrel{\ref{lem:ion}}{=}& \stab_\ast\left(-(x_2+x_5)(x_1\widetilde{\psi}_1)\cap\left[\relrub{0}{x_1,x_2+x_5,x_3,x_4)}\right]^{\rm vir}\right)\\
\stackrel{\ref{lem:cor}}{=}& \stab_\ast\left(-x_1(x_2+x_5)\stab^\ast\psi_1\cap\left[\relrub{0}{x_1,x_2+x_5,x_3,x_4)}\right]^{\rm vir}\right)\\
=&-x_1(x_2+x_5)
\end{align*}

Thus in chamber $\fc_2$ we deduce that 
\begin{align*}
H_0(\bx)=&3!(x_1^2 + 2x_1\psi_1\stab_\ast \sD_{\rm vir} +\stab_\ast(\sD_{\rm vir}^2))\\
=&3!(x_1^2+2x_1(x_2+x_5) - x_1(x_2+x_5))\\
=&3!x_1(x_1+x_2+x_5)
\end{align*}
as expected.  Notice that computing the polynomial wall--crossing contribution of Theorem~\ref{thm:wc} in this way produces the expected result as well: 
\begin{align*}
WC_{\left\{x_2,x_5\right\}}(\bx)=&3!(2x_1\psi_1\stab_\ast \sD_{\rm vir} +\stab_\ast(\sD_{\rm vir}^2))\\
=&3!(2x_1(x_2+x_5) - x_1(x_2+x_5))\\
=&3!x_1(x_2+x_5).
\end{align*}
 
\begin{figure}[tb]
\begin{center}
		\includegraphics[width=0.6\textwidth]{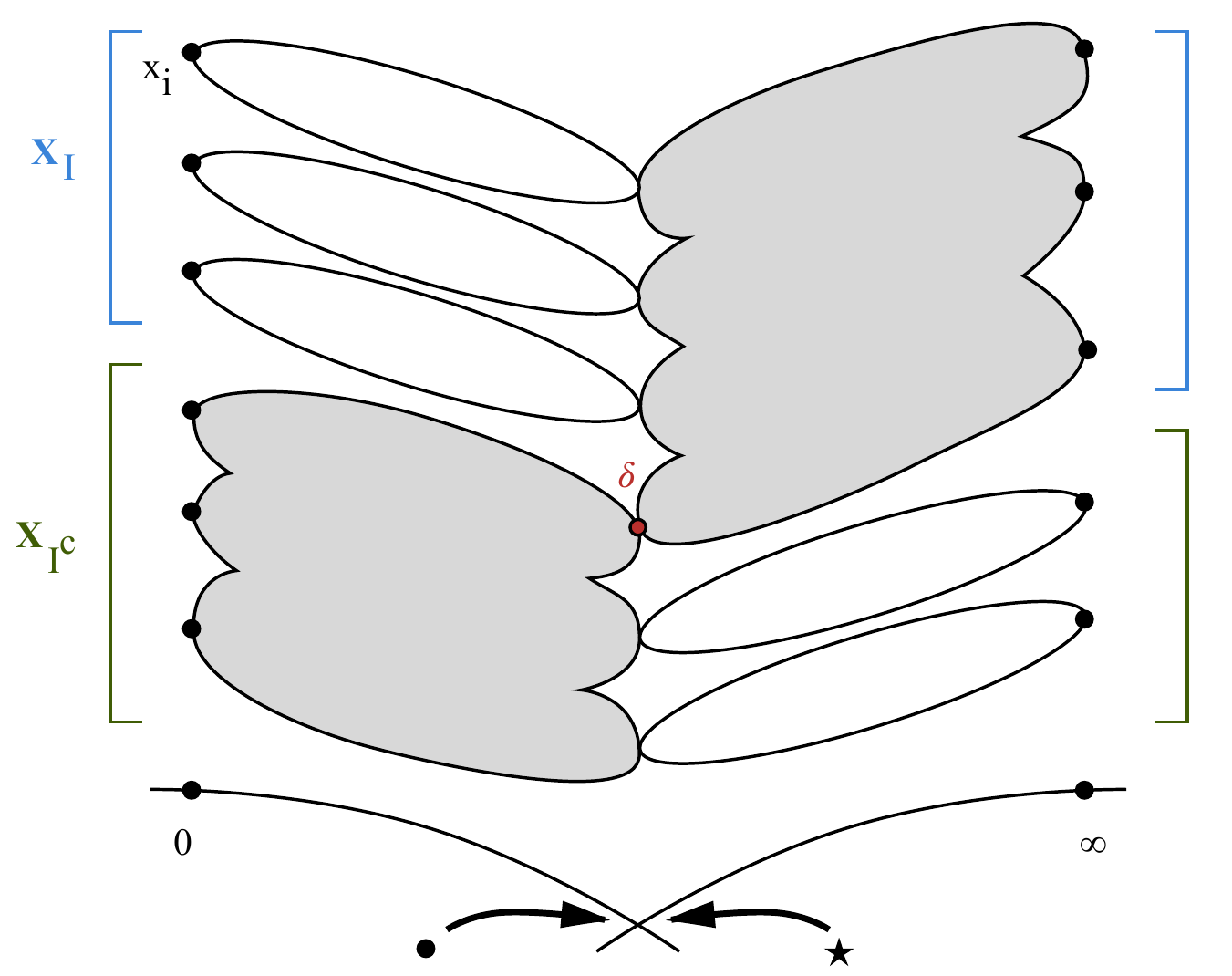}
\end{center}
\caption{The special irreducible component $D$ of the genus $0$ wall--crossing divisor $\sD_i^{W_I}$ for a proper subset $I$ of $\left\{1,\ldots,n\right\}$ containing $i$.  This component exists only when $\sum_{j\in I}x_j <0$. The two non--trivial components meet at the solitary non--trivial node with multiplicity $\delta=-\sum_{j\in I}x_j$.}  
\label{fig:genus0}
\end{figure}

\subsection{A genus zero curiosity}
We conclude our investigation with an intriguing observation in genus $0$. Here, every time one crosses a wall $W_I=\{\sum_{i\in I}x_i=0\}$, the wall crossing divisor $\sD^{W_I}_i$ contains exactly one irreducible component $D$ which does not pushforward to $0$ via $\stab$, depicted in Figure~\ref{fig:genus0}. The divisor $D$ has no transverse part  to its self intersection, and therefore $D^2= -D(\br^\ast(\widehat\psi_\bullet) +\br^\ast(\widehat\psi_\star))$. Together with Lemma \ref{lem:side}, this gives the impression that one may easily obtain a geometric proof of the combinatorial formula for the wall crossing (\cite{ssv:gz}), by using $D$ as the only variation of the  $\tilde{\psi}$ class across the wall. This approach is incorrect: pushforward doesn't commute with intersection, and hence the other irreducible  components of $\sD_i^{W_I}$ can and in fact do contribute to higher powers of $\sD_i^{W_I}$. However, if one ignores all such contributions, the final result is only incorrect up to sign. We have no conceptual explanation of this phenomenon, but we take it to mean that there are interesting and potentially useful vanishing statements hidden in the intersections that constitute the geometric wall crossing formula \eqref{for:wc}. We make this precise in the following proposition.
 
\begin{prop}\label{prop:gz}
Let $g=0$ and consider two adjacent chambers $\fc_1$ and $\fc_2$, separated by a wall $W_{I}= \{\sum_{i\in I} x_i=0\}$. Assume $\bx\in \fc_2 = \{\sum_{i\in I} x_i<0\}$ and let $D$  be the irreducible divisor depicted in Figure~\ref{fig:genus0}, with $r_1$ and $r_2$ being the number of simple ramification on each component of a cover in $D$. We define the ``naive" psi class in chamber $\fc_1$ as $\tilde{\tilde{\psi}}_i:=\tilde\psi_i-D$.  Then
\beq
\label{eq:nwc}
r! \stab_\ast((x_i\tilde\psi_i)^{r-1}-(x_i\tilde{\tilde{\psi}}_i)^{r-1})= (-1)^{r_1-1}WC_I(\bx).
\eeq
\end{prop}
\begin{proof} We evaluate the intersection in parenthesis in the left hand side of \eqref{eq:nwc}:
\begin{align}
\nonumber (x_i{\tilde{\psi}_i})^{r-1}-(x_i{\tilde{\psi}_i}-D)^{r-1}&=  -\sum_{k=1}^{r-1}{{r-1}\choose{k}}(-1)^kD^k(x_i{\tilde{\psi}_i})^{r-1-k}\\
\nonumber & =  \sum_{k=1}^{r-1}{{r-1}\choose{k}}(-1)^{r-1-k}D\br^\ast(\widehat\psi_\bullet)(\br^\ast(\widehat\psi_\bullet) +\br^\ast(\widehat\psi_\star))^{k-1}\\
 & =  \left(\sum_{k=r_2}^{r-1}{{r-1}\choose{k}}{{k-1}\choose{r_2-1}} (-1)^{r-1-k} \right)D\br^\ast(\widehat\psi_\bullet)^{r_1-1}\br^\ast(\widehat\psi_\star)^{r_2-1} \label{gzero}
\end{align}
\end{proof}
We make the claim that the summation of those combinatorial coefficients is just a fancy way of saying $(-1)^{r_1-1}$. Then, multiplying \eqref{gzero} by $r!$ and pushing forward via $\stab$, we obtain
\beq
(-1)^{r_1-1}\delta {{r-1}\choose{r_1}}H_0(x_I+\delta)H_0(x_{I^c}-\delta)= (-1)^{r_1-1}WC_I(\bx).
\eeq

To prove the claim, it suffices to observe that the summation in parenthesis in \eqref{gzero} equals
\beq
r_2{{r-1}\choose{r_2}}\int_0^1t^{r_2-1}(t-1)^{r_1-1}dt.
\eeq
Integrating by parts $r_1-1$ times evaluates this integral at $(-1)^{r_1-1}$. 

\bibliographystyle{amsalpha}             
\bibliography{geomWC}       
\end{document}